\documentclass[reqno,a4paper, 11pt]{amsart}

\usepackage[a4paper=true,pdfpagelabels]{hyperref}
\usepackage{graphicx}

\usepackage[ansinew]{inputenc}
\usepackage{amsfonts,epsfig}
\usepackage{latexsym}
\usepackage{amsmath}
\usepackage{amssymb}
\usepackage{mathabx}
\usepackage{comment}
\usepackage{color}

\newtheorem{theorem}{Theorem}

\newtheorem{corollary}[theorem]{Corollary}

\newtheorem*{question}{Question}

\theoremstyle{definition}

\newtheorem{example}[theorem]{Example}

\theoremstyle{remark}

\numberwithin{equation}{section}

\newcommand{\D}{\mathbb{D}}

\newcommand{\N}{\mathbb{N}}

\newcommand{\C}{\mathbb{C}}

\newcommand{\e}{\varepsilon}

\newcommand{\I}{\mathcal{I}_c}

\newcommand{\Th}{\Theta}

\def\a{\alpha}       \def\b{\beta}        \def\g{\gamma}
\def\d{\delta}           \def\e{\varepsilon}
     \def\om{\omega}      
\def\s{\sigma}       \def\t{\theta}       
         \def\r{\rho}         
                  \def\vp{\varphi}

\DeclareMathOperator{\supp}{supp}
\DeclareMathOperator{\dist}{dist}
\DeclareMathOperator{\sing}{sing}

\begin{document}

\title{A Characterization of One-component Inner Functions}


\thanks{The first author is supported in part by the Generalitat de Catalunya (grant 2017 SGR 395) and the Spanish Ministerio de Ciencia e Innovacion (project  MTM2017-85666-P). The second author is supported in part by the Saastamoinen Foundation and the Academy of Finland (project 286877).}


\author{Artur Nicolau}
\address{Universitat Aut\`onoma de Barcelona, Departament de Matem\`atiques,  08193 Barcelona, Catalonia}
\email{artur@mat.uab.cat}

\author{Atte Reijonen}
\address{University of Eastern Finland, P.O.Box 111, 80101 Joensuu, Finland}
\email{atte.reijonen@uef.fi}

\maketitle

\begin{abstract}
We present a characterization of one-component inner functions in terms of the location of their zeros and their associated singular measure. As consequence we answer several questions posed by J.~Cima and R.~Mortini. In particular we prove that for any inner function $\Theta$ whose singular set has measure zero, one can find a Blaschke product $B$ such that $\Theta B$ is one-component. We also obtain a characterization of one-component singular inner functions which is used to produce examples of discrete and continuous one-component singular inner functions. 
\end{abstract}

\section{Introduction and main results}\label{Sec1}

Let $\D$ be the open unit disc of the complex plane and let $\partial \D$ be the unit circle. An inner function is a bounded analytic function in $\D$
having unimodular radial limits almost everywhere on $\partial \D$. It is a classical result that any inner function can be factorized as the product of a Blaschke product, a singular inner function and unimodular constant (\cite{Garnett1981}).
Recall that, for a given sequence $\{z_n\}\subset\D$ satisfying $\sum_n (1-|z_n|)<\infty$, the Blaschke product
with zeros $\{z_n\}$ is defined by
    \begin{equation*}\label{Eq:Blaschke}
    B(z)=\prod_{n}\frac{|z_n|}{z_n}\frac{z_n-z}{1-\overline{z}_nz}, \quad z\in \D.
    \end{equation*}
Here each zero is repeated according to its multiplicity and the convention $|z_n|/z_n=1$ is used when $z_n=0$.
A singular inner function is an inner function of the form
    $$
    S(z)=\exp\left(\int_{\partial \D} \frac{z+\xi}{z-\xi}\, d\s(\xi) \right),\quad z\in\D,
    $$
where $\s$ is a positive measure on $\partial \D$, singular with respect to the Lebesgue measure. The singular set of an inner function $\Th$, which will be denoted by $\sing \Th$, consists of all points on $\partial \D$ in which $\Th$ does not have an
analytic continuation. If $\Th$ factors as $\Th = \lambda BS$, where $|\lambda|=1$, $B$ is a Blaschke product and $S$ is the singular inner function associated to the singular measure $\sigma$, then $\sing \Th$ is precisely the union of the accumulation points of zeros of $\Th$ and the closed support of the measure $\sigma$. See Chapter II of \cite{Garnett1981}.  

We focus on so-called one-component inner functions introduced by B. Cohn in  \cite{Cohn1982}, which are inner functions $\Th$ whose level set
    $\{z\in \D:|\Th(z)|<\e\}$
is connected for some $0< \e < 1$. For simplicity, we denote by $\I$ the set of all one-component inner functions. 
The main motivation to study $\I$ comes from the theory of model spaces $K_\Th^p=H^p \cap \overline{z}\Th \overline{H^p}$, $1<p<\infty$,  generated by the inner function $\Th$. 
For instance, B.~Cohn characterized Carleson measures for $K_\Th^2$ when $\Th \in \I$ in terms of their action on reproducing kernels (\cite{Cohn1982}),
and then S.~Treil and A.~Volberg generalized Cohn's result to all $p\in (1,\infty)$ (\cite{VT1988}).  It is also worth mentioning that N. Nazarov and A. Volberg proved that Cohn's result does not hold for arbitrary inner functions. See \cite{NV2002}. The class $\I$ also appears naturally in several recent results in the context of operator theory in $K_\Th^2$ \cite{ALMP2016, BBK2011, Besonov2016, Besonov2015}.

A.~Aleksandrov obtained a series of nice descriptions of inner functions in $\I$ in terms of the behaviour of their derivatives (\cite{Aleksandrov2002}). As a byproduct he proved  also a strong form of the Schwarz-Pick lemma for inner functions in $\I$. Using Aleksandrov's descriptions, J.~Cima, R.~Mortini and the second author constructed some concrete examples of one-component inner functions \cite{CM2017,CM2019,R2019}. In particular singular inner functions associated to a finite sum of weighted Dirac masses are one-component and thin Blaschke products are not in $\I$. A Blaschke product $B$ whose zeros $\{z_n\}_{n=1}^\infty$ (ordered by non-decreasing moduli) lie in a Stolz angle is one-component if
    \begin{equation*}
    \begin{split}
    \liminf_{n \rightarrow \infty}\frac{\sum_{|z_j|>|z_n|} (1-|z_j|)}{1-|z_n|}>0.
    \end{split}
    \end{equation*}
The main motivation of our work is to give descriptions and examples of one-component inner functions in terms of the location of their zeros and their associated singular measures. As it will be explained, our results provide answers to several questions posed by  J.~Cima and R.~Mortini.   

Let $\Th$ be an inner function which factors as $\Th = \lambda BS$, where $|\lambda|=1$, $B$ is a Blaschke product with zeros $\{z_n \}$ and $S$ is the singular inner function associated to the singular measure $\sigma$. Consider the measure $\mu(\Th)$ on $\overline{\D}$ defined as
	$$ \mu(\Th) = \sum_n  (1-|z_n|) \d_{z_n} + \s.$$
Here $\d_{z}$ denotes the Dirac point measure at the point $z$. We will describe one-component inner functions $\Th$ in terms of the mass given by $ \mu(\Th)$ to Carleson squares defined as 
	$$Q(z)=\{w\in \overline{\D}: |\arg z - \arg w|\le (1-|z|)/2, |w|\ge |z|\},  \quad  z \in \D . $$ 
This idea originates in \cite{Bishop1990} where C.~Bishop described inner functions $\Th$ in the little Bloch space in terms of the behaviour of the corresponding measure $ \mu(\Th)$. Similar ideas have been used in \cite{Nicolau1994} and \cite{Mortini-Nicolau2004}. 

\begin{theorem}\label{Thm1}
Let $\Th$ be an inner function. Then $\Th\in \I$ if and only if there exists a constant  $C=C(\Th)$ with $0<C<1$ such that $
\mu(\Th) (Q(z))=0$ when $|\Th(z)|\ge C$.
\end{theorem}

This result can be proved applying \cite[Theorem~1.2]{Aleksandrov2002} by A.~Aleksandrov but in Section~\ref{Sec3}, we present a  self-contained proof relying on Hall's lemma and a stopping time argument. Let us say that an inner function is a finite-component inner function if it has a finitely connected level set.
Applying the proof of Theorem~\ref{Thm1}, we show in Section~\ref{Sec3} that all finite-component inner functions belong to $\I$.

\begin{corollary}\label{Coro1}
Let $\Th$ be an inner function. If there exists a constant $C$ with $0<C<1$ such that $\{z\in \D: |\Th(z)|<C\}$ is finitely connected, then $\Th \in \I$.
\end{corollary}

For $1<\a<\infty$, let 
    $\Gamma_\a(e^{i\t})=\left\{z\in\D:|z-e^{i\t}|<\a(1-|z|)\right\}$
denote the Stolz angle with vertex at $e^{i\t}\in \partial \D$.
As another consequence of Theorem~\ref{Thm1}, in Section~\ref{Sec3}, we characterize one-component Blaschke products whose zeros are contained in a Stolz angle.

\begin{corollary}\label{Coro1b}
Let $1<\a<\infty$ and let $B$ be a Blaschke product with infinitely many zeros. Assume that the zeros of $B$ are contained in a Stolz angle with vertex at $e^{i\t}\in \partial \D$. Then $B\in \I$ if and only if $\limsup_{r \rightarrow 1^-}|B(re^{i\t})|<1$.
\end{corollary}

A sequence of points $\{z_n\}\subset \D$ is called  uniformly separated if
    $$
    \inf_{n\in\N}\prod_{k\ne n}\left|\frac{z_k-z_n}{1-\overline{z}_kz_n}\right|>0.
    $$
A celebrated result by L. Carleson says that interpolating sequences for the algebra of bounded analytic functions in $\D$ are precisely the uniformly separated sequences. A Blaschke product with uniformly separated zeros is called an interpolating Blaschke product. Given a measurable set $E \subset \partial \D$ let $|E|$ denote its Lebesgue measure.  

One-component inner functions have many special properties. For instance the singular set of an inner function in $\I$ has Lebesgue measure zero. See \cite{Aleksandrov1989}. Theorem~\ref{Thm2} below gives an affirmative answer to the following question posed in \cite{CM2017} by J.~Cima and R.~Mortini:
\emph{Can every inner function $\Th$ with $|\sing \Th|=0$ be multiplied by a one-component inner function $B$
into $\I$?} In addition, we show that $B$ can be chosen to be an interpolating Blaschke product. 

\begin{theorem}\label{Thm2}
Let $\Th$ be an inner function whose singular set has Lebesgue measure zero. Then there exists an interpolating Blaschke product $B \in \I$ such that $B\Th \in \I$.
\end{theorem}

Theorem~\ref{Thm2} implies also a negative answer to the following question posed in \cite{CM2019} by J.~Cima and R.~Mortini:
\emph{Is the singular set of any inner function in $ \I$ necessarily countable?} 

The main effort in the proof of Theorem~\ref{Thm2} is to find a suitable Blaschke product $B$. Once the zeros of $B$ are well located, the assertion can be proved quite easily by using Theorem~\ref{Thm1}.
Roughly speaking, $B$ is chosen such that $\sing B=\sing \Th$, zeros of $B$ are close enough to $\partial \D$ (depending on $\Th$) and form a chain where the  pseudohyperbolic distance of adjacent points is fixed. The detailed proof is presented in Section~\ref{Sec4}.

Recall that an analytic function $f$ in $\D$ belongs to the Nevanlinna class $\mathcal{N}$ if
     $$\sup_{0 < r<1} \int_0^{2\pi} \log^+ |f(re^{i\t})|\,d\t<\infty,$$
where $\log^+ 0=0$ and $\log^+ x=\max\{0,\log x\}$ for $0<x<\infty$. Deep results on inner functions whose derivative is in the Nevanlinna class have been recently obtained by O. Ivrii. See \cite{Ivrii2017} and \cite{Ivrii2018}. 
Applying Theorem~\ref{Thm2}, we deduce that some one-component inner functions might be bad-behaving in several ways. As an example, we show in Section~\ref{Sec4} that
$\I$ is not contained in $\{f: f'\in \mathcal{N}\}$.

\begin{corollary}\label{Coro2}
There exists $\Th \in \I$ such that $\Th' \notin \mathcal{N}$.
\end{corollary}

In Section~\ref{Sec5} we study one-component singular inner functions. As another application of Theorem 1, we present the following characterization of one-component singular inner functions.  

\begin{theorem}\label{singular}
Let $S$ be a singular inner function associated with a non-trivial singular measure~$\s$. Consider the set $\Omega = \{z\in \D: 1-|z|\ge 2\dist (z/|z|, \supp \s)\}$.
Then $S \in \I$ if and only if $\limsup_{z\in \Omega, |z|\rightarrow 1^-} |S(z)|<1$.
\end{theorem}

Note that the set $\Omega$ in the statement is a sawtooth region, as the following figure shows.
\begin{figure}[h!]
\label{pic}
\begin{center}
\includegraphics[width=10cm]{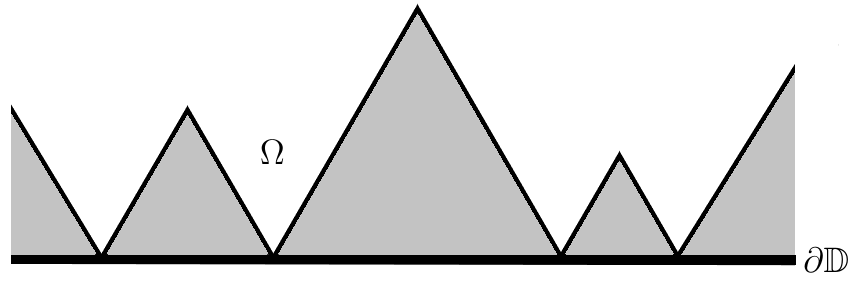}
\caption{The region $\D \setminus \Omega$, painted in grey, consists of a union of tents whose bottoms are the complementary arcs of $\supp \s$ in $\partial \D$. Each tent has the same shape, but size depends on the length of the complementary arc. }
\end{center}
\end{figure}

Using Theorem~\ref{singular}, we show that any singular inner function associated to a Cantor measure in a symmetric Cantor set is one-component. On the other hand, also as an application of Theorem~\ref{singular}, we construct discrete measures whose associated singular inner functions are not one-component. Theorem~\ref{singular} also implies the following result which provides an  affirmative answer to another question of J.~Cima and R.~Mortini (Question 4.4 i) in \cite{CM2019}).

\begin{corollary}\label{countable}
Let $E$ be a closed countable set of the unit circle. Let $\s$ be a positive singular measure supported on $E$ such that $\s (\{\xi\})>0$ for any $\xi \in E$. Then the singular inner function associated to $\s$ is one-component. 
\end{corollary}

We finish the paper in Section~\ref{Sec5} with an open question.

\section{Proofs of Theorem~\ref{Thm1}, Corollaries~\ref{Coro1}~and~\ref{Coro1b}}\label{Sec3}

We first fix some notation. Consider the  pseudohyperbolic distance $\r (z,w)$ between the points  $z,w\in \D$ given by  $\r(z,w)=|(z-w)/(1-\overline{z}w)|$. 
A sequence of points $\{z_n\}\subset \D$ is called separated if there exists $\d=\d(\{z_n\})\in (0,1)$ such that $\r(z_j,z_k)>\d$ for all distinct $j$ and $k$.
Moreover, we write $f \lesssim g$ if there exists an absolute constant $C>0$ such that $f\le C g$, while $f \gtrsim g$ is understood analogously.
If $g \lesssim f \lesssim g$, then the notation $f \asymp g$ is used.

\medskip

\noindent
\emph{Proof of Theorem~\ref{Thm1}.} Let us first prove the necessity.
Assume contrary that there exists a sequence of points $\{z_n\}_{n=1}^\infty \subset \D$ such that  $\mu(\Th) (Q(z_n))>0$ and $|\Th(z_n)| \rightarrow 1^-$ as $n \rightarrow \infty$. Since $\Th\in \I$, we find $C\in (0,1)$ such that the level set $\Omega=\{z\in \D : |\Th(z)|\le C\}$ is connected. Given a domain $A \subset \C$ and a subset $E \subset \partial A$, let $\om(z,E,A)$ be the harmonic measure of $E$ in the domain $A$, that is, the harmonic function in $A$ whose boundary values are identically $1$ on $E$ and identically zero on $\partial A \setminus E$. Since $\Th$ is inner,  the subharmonicity of $\log |\Th(z)|$  gives that $\log |\Th (z)| \leq (\log C ) \, \om(z,\partial \Omega,\D \setminus \Omega)$ for any $z \in \D \setminus \Omega$. Then by  Hall's lemma (see \cite[Chap.~12, Lemma~4]{Duren1970}), there exists an absolute constant $C_1>0$ such that
    \begin{equation}\label{Eq:Thm1-1}
    \begin{split}
    \log |\Th(z)|\le 
    C_1 (\log C) \, \om(z,(\partial \Omega)^*,\D), \quad z\in \D \setminus \Omega,
    \end{split}
    \end{equation}
where $(\partial \Omega)^*=\{z/|z|: z\in \partial \Omega\}$ is the radial projection of $\partial \Omega$. Since $\mu(\Th) (Q(z_n))>0$ and the diameter of the connected set  $\Omega$ depends only on $\Th$ and $C$,
we find $N\in \N$ such that $|(\partial \Omega)^* \cap 2Q(z_n)|\ge (1-|z_n|)/2$ for $n \ge N$. Consequently,
    \begin{equation}\label{Eq:Thm1-2}
    \begin{split}
    \om(z_n ,(\partial \Omega)^*,\D)&=\int_{(\partial \Omega)^*} \frac{1-|z_n|^2}{|e^{i\t}-z_n|^2}\, \frac{d\t}{2\pi}
    \ge \int_{(\partial \Omega)^* \cap 2Q(z_n)} \frac{1-|z_n|^2}{|e^{i\t}-z_n|^2}\, \frac{d\t}{2\pi} \\
    &\asymp \frac{|(\partial \Omega)^* \cap 2Q(z_n)|}{1-|z_n|}\ge \frac{1}{2}, \quad n \ge N.
    \end{split}
    \end{equation}
Combining estimates~\eqref{Eq:Thm1-1}~and~\eqref{Eq:Thm1-2}, we deduce that there exists a positive constant $D<1$ such that $|\Th(z_n)| \le D$ for all $n \ge N$.
This contradiction finishes the proof of the necessity. 

Next we prove the sufficiency. Pick a constant $C_1\in (0,1)$ such that $(1-C_1)/(1-C)$ is very small. In particular, we assume $C_1>(C+9/10)/(1+9C/10)$. Consider the decomposition of $\D$ into dyadic Carleson squares
	$$Q_{n,k}=\left\{re^{i\t}: 1-\pi 2^{-n}\le r<1, 2\pi k2^{-n}\le \t<2\pi(k+1)2^{-n} \right\},$$
where $n\ge 2$ and $0\le k < 2^{n}$. Let  $\mathcal{T}(Q_{n,k}) =\{z\in Q_{n,k}:|z|\le 1-\pi 2^{-n-1}\}$ denote the top half of $Q_{n,k}$.  Let $\mathcal{G}=\{Q_j\}$ be the collection of maximal dyadic Carleson squares  such that
	$\sup_{z\in \mathcal{T}(Q_{j})} |\Th(z)|\ge C_1$.
If $w \in \D$ satisfies $\r(w,\mathcal{T}(Q_j)) \le 9/10$ for some $j$, then $|\Th(w)|\ge C$.
This is easy to deduce from the estimate
    \begin{equation*}
    \begin{split}
    \r(w,\mathcal{T}(Q_j))\ge \inf_{z\in Q_j}\left|\frac{\Th(z)-\Th(w)}{1-\overline{\Th}(z)\Th(w)}\right|\ge  \inf_{z\in Q_j}\frac{|\Th(z)|-|\Th(w)|}{1-|\Th(z)||\Th(w)|},
    \end{split}
    \end{equation*}
where the first inequality is due to the Schwarz-Pick lemma. Thus our hypothesis implies $\mu(\Th) (Q(w))=0$ when $w$ is as above.
In particular, $\mu (\Th) (2Q_j)=0$ for all $j$.

Let $z_j$ be the center of $\mathcal{T}(Q_j)$. By the Schwarz-Pick lemma, we find $C_2=C_2(C_1)\in (0,1)$ with $C_2=C_2(C_1) \rightarrow 1^-$ as $C_1 \rightarrow 1^-$, 
such that $|\Th(z_j)|\ge C_2$ for all $j$.
Next we show that there exists a universal constant $C_3 >0$ such that
	 \begin{equation}\label{Eq:Thm1-3}
   |\Th(z)|\ge |\Th(z_j)|^{C_3}\ge C_2^{C_3}, \quad z\in Q_j.
    \end{equation}
Applying the fact that $\mu (\Th) (2Q_j)=0$ for all $j$, we obtain
	  \begin{equation*}
    \begin{split}
    \log |\Th(z)|^{-1}& \asymp \int_{\overline{\D}} \frac{1-|z|^2}{|1-\overline{w}z|^2} d\mu (\Th) (w)=\int_{\overline{\D}\setminus 2Q_j} \frac{1-|z|^2}{|1-\overline{w}z|^2} \,d\mu (\Th)(w) \\
    &\asymp \frac{1-|z|^2}{1-|z_j|^2}\int_{\overline{\D}\setminus 2Q_j} \frac{1-|z_j|^2}{|1-\overline{w}z_j|^2}\, d\mu (\Th)(w) \\
    &\asymp \frac{1-|z|}{1-|z_j|} \log |\Th(z_j)|^{-1}, \quad z\in Q_j.
    \end{split}
    \end{equation*}
Since $1-|z| \lesssim 1-|z_j|$, estimate \eqref{Eq:Thm1-3} holds.

Set $\Omega_1=\{z\in \D: |\Th(z)|<C_1\}$.  We show that $\Omega_1$ is connected arguing by contradiction. Let $G $ be the union of all Carleson squares in the family $\mathcal{G}$. 
First we note that $\D \setminus G \subset \Omega_1$ by the construction of $\mathcal{G}$. 
Since $\D \setminus G$ is connected, we find a connected component $\Omega_2$ of $\Omega_1$ with $\D \setminus G \subset \Omega_2$.
Assume that $\Omega_3$ is a connected component of $\Omega_1$ satisfying $\Omega_2 \cap \Omega_3 = \emptyset$. In particular,
$\Omega_3 \subset G$. Hence estimate \eqref{Eq:Thm1-3} gives
    \begin{equation}\label{Eq:Thm1-4}
   \Omega_3\subset \{z\in \D: |\Th(z)|\ge C_2^{C_3}\}.
   \end{equation}
By the maximum principle $\Omega_3$ is simply connected and we can consider a conformal mapping $\vp:\D \rightarrow \Omega_3$.
Then $g=C_1^{-1}\Th \circ \vp$ is an inner function. This fact was noted in the proof of \cite[Corollary~1.2]{Cohn1982}
and it follows essentially from \cite[Theorem~VIII.~31]{Tsuji1975}.
Now \eqref{Eq:Thm1-4} implies $g \equiv 1$, which
is a contradiction. Thus $\Omega_1$ is connected and the proof is complete. \hfill$\Box$

\medskip

Next we prove Corollaries~\ref{Coro1}~and~\ref{Coro1b}.

\medskip

\noindent
\emph{Proof of Corollary~\ref{Coro1}.}
Assume contrary that $\Th$ is not a one-component inner function. By Theorem 1, there exists a sequence  $\{z_j\}_{j=1}^\infty \subset \D$ such that  $\mu(\Th) (Q(z_n))>0$ for every $n$ and $|\Th(z_j)| \rightarrow 1^-$ as $j \rightarrow \infty$.
By the assumption we have $\{z\in \D: |\Th(z)|<C\}=\bigcup_{n=1}^N \Omega_n$ where  $\Omega_n \subset \D$ are connected sets.
Then $Q(z_j)\cap \Omega_k \neq \emptyset$ for some $k\in \{1,\ldots , N\}$ and all $j\in \N$. Applying Hall's lemma, we obtain
	\begin{equation*}
    \begin{split}
    \log |\Th(z)|\le \log C\, \om(z,\partial \Omega,\D \setminus \Omega_k)\le \log C\, \om(z,(\partial \Omega_k)^*,\D), \quad z\in \D \setminus \Omega_k,
    \end{split}
    \end{equation*}
where the notation is same as in \eqref{Eq:Thm1-1}. Again arguing as in the proof of Theorem~\ref{Thm1}, one can find   a constant $D<1$ such that
 $|\Th(z_j)| \le D$ for all $j$ sufficiently large. Since this is a contradiction, the assertion is proved. \hfill$\Box$

\medskip

\noindent
\emph{Proof of Corollary~\ref{Coro1b}.}
We can assume $e^{i \theta} = 1$. Let $0<\b<\pi/2$ and let $\Gamma_\b$ be a cone in $\D$ with aperture $2\b$ and vertex at 1.
Assume without loss of generality that the zeros of $B$ are contained in $\Gamma_\b$.

Let us first prove the necessity. Assume contrary that exists a sequence $\{r_n\}_{n=1}^\infty\subset (0,1)$ such that 
$|B(r_n)|\rightarrow 1^-$ as $n\rightarrow \infty$. 
Since $\mu(B)(Q(r_n))>0$ for any $n\in \N$, Theorem~\ref{Thm1} implies $B\notin \I$. This is a contradiction and the necessity is proved.

Next we prove the sufficiency. Let us  define $\Gamma_\g$ in a similar way as $\Gamma_\b$, and choose $\g=\g(\b)$, $\b<\g<\pi/2$, such that $Q(z)$ is contained in $\overline{\D}\setminus \Gamma_\b$ when $z\in \D\setminus \Gamma_\g$.
\begin{figure}[h!]
\label{fic2}
\begin{center}
\includegraphics[width=9cm]{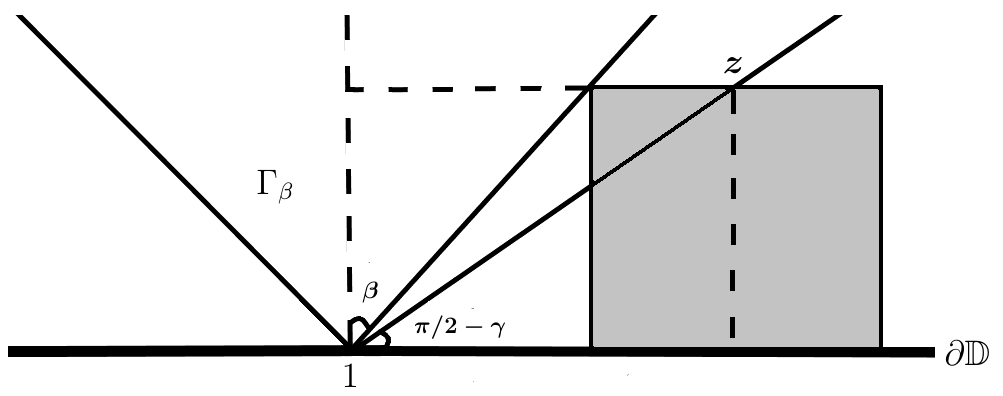}
\caption{The picture shows the cone $\Gamma_\b$ and the Carleson square $Q(z)$.}
\end{center}
\end{figure}

\noindent
Using the Schwarz-Pick lemma together with the assumption $\limsup_{r\rightarrow 1^-}|B(r)|<1$, we find a constant $C=C(\g)<1$ such that $\limsup_{z\in \Gamma_\g, z\rightarrow 1^-}|B(z)|\le C$. By Theorem~\ref{Thm1} it suffices to show that $\mu(B) (Q(z))=0$ for $z\in \D\setminus \Gamma_\g$.
Since this follows from the choice of $\g$, the proof is complete.   
\hfill$\Box$

\section{Proofs of Theorem~\ref{Thm2} and Corollary~\ref{Coro2}}\label{Sec4}

We go directly to the proofs.

\medskip

\noindent
\emph{Proof of Theorem~\ref{Thm2}.} We first consider a Whitney type decomposition of the open set  $\partial \D \setminus \sing \Th= \bigcup_{n=1}^\infty I_n$, where each $I_n$ is a closed arc on $\partial \D$
satisfying $|I_n| \asymp \dist (I_n, \sing \Th)$. 

Next we construct the Blaschke product $B$. Fix a sequence $\{\e_n\}\subset (0,1)$ with $\lim_{n \to \infty} \e_n=0$. 
Let us choose  $r_n\in (0,1)$ such that,  if $|z|\ge r_n$ and $e^{i \arg z}\in I_n$, then  $|\Th(z)|\ge 1-\e_n$. 
Let $\Gamma$ be the curve containing $\bigcup_{n=1}^\infty \{r_n\xi: \xi \in I_n\}$ and the radial segments connecting arcs $\{r_n\xi: \xi \in I_n\}$ in the natural way.
Now locate the zeros $\mathcal{Z}(B)=\{z_j\}_{j=1}^\infty$ of $B$ on the curve $\Gamma$ such that, for each $z_j$, there exist distinct zeros $z_m, z_l$ satisfying $\r(z_j,z_m)=\r(z_j,z_l)=1/10$,
while other zeros are further away from $z_j$. In other words, place the zeros of $B$ in the curve $\Gamma$ at each $1/10$ pseudohyperbolic units. 

\begin{figure}[h!]
\label{pic}
\begin{center}
\includegraphics[width=10cm]{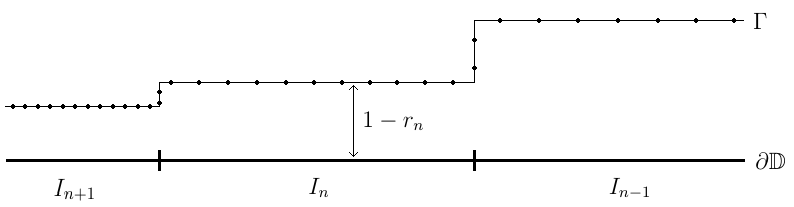}
\caption{The picture shows the curve $\Gamma$ where the zeros of $B$, represented by points, are located.}
\end{center}
\end{figure}

Let us recall that a sequence $\{z_j\}$ is uniformly separated if and only if $\{z_j\}$ is separated
and $\sum_{z_j\in Q(z)} (1-|z_j|) \lesssim 1-|z|$ for all $z\in \D$ (see \cite[Chap.~VII,~Theorem~1.1]{Garnett1981}). Using this result together with the fact that the pseudohyperbolic distance of adjacent
points in $\mathcal{Z}(B)$ is $1/10$, one checks that $B$ is an interpolating Blaschke product.

Let $0 < C < 1$ be a constant close to $1$ to be fixed later. Since we will prove the assertion by Theorem~\ref{Thm1}, points $z$ where $|B(z)|\le C$ are not relevant.
In particular, this is the case in $\Omega=\bigcup_{n=1}^\infty \{r\xi: r\le r_n, \xi\in I_n\}$, as the following argument shows. Fix $z \in \Omega$.  If $\r(z, \mathcal{Z}(B))\le 1/2$,
then $|B(z)|\le \r(z, \mathcal{Z}(B))\le 1/2$ by Schwarz's Lemma.
Hence we may assume that $z\in \Omega$ satisfies $\r(z, \mathcal{Z}(B))>1/2$. Then standard estimates give
    \begin{equation*}
    \begin{split}
    \log |B(z)|^{-1} &= \sum_{j=1}^\infty \log \r(z,z_j)^{-1} \asymp\sum_{j=1}^\infty\left(1-\r(z,z_j)^2\right) \\
    &=\sum_{j=1}^\infty \frac{(1-|z|^2)(1-|z_j|^2)}{|1-\overline{z}_jz|^2} \\
    &\gtrsim  \sum_{\{j :|z_j-z|\le 2(1-|z|)\}} \frac{1-|z_j|}{1-|z|} \gtrsim 1,
    \end{split}
    \end{equation*}
where the last inequality is due to the location of zeros $z_j$.
Consequently, there exists a constant $C<1$ such that $|B(z)|\le C$ for any $z \in \Omega$. 

Let $z \in \D \setminus \Omega$ and assume $|B(z)|> 12/21$. If $w\in \Gamma$, then the pseudohyperbolic triangle inequality, the Schwarz-Pick lemma
and the location of zeros give
    \begin{equation*}
    \begin{split}
    \r(z,w)\ge \inf_{v\in \mathcal{Z}(B)}\frac{\r(z,v)-\r(v,w)}{1-\r(z,v)\r(v,w)} \ge \frac{|B(z)|-1/10}{1-|B(z)|/10}>\frac{1}{2}.
    \end{split}
    \end{equation*}
It is easy to check that $w$ satisfies $2|z-w|> 1-|z|$ if $\r(w,z)>1/2$. This implies that $\mu (B) (Q(z))=0$ when $|B(z)|>12/21$. 
Hence $B$ belongs to $\I$ by Theorem~\ref{Thm1}. 

By the choice of $\Gamma$, it is obvious that possible zeros of $\Th$ are in $\Omega$. Hence the previous argument also applies to $B\Th $ and again by Theorem~\ref{Thm1}, $B \Th $ is a one-component inner function. This completes the proof. \hfill$\Box$

\medskip

\noindent
\emph{Proof of Corollary~\ref{Coro2}.}
Let $\Th$ be a Blaschke product whose singular set has measure zero such that $\Th'\notin \mathcal{N}$. See  \cite[Theorem~4]{Pelaez2008}. By our Theorem~\ref{Thm2}, we can  find a Blaschke product $B$
such that $|\sing B|=0$ and  $B\Th \in \I$.
Now it suffices to deduce $(B\Th)' \notin \mathcal{N}$ by applying the following consequence of \cite[Theorem~2~and~Corollary~4]{AhernClark1974}:
The derivative of a Blaschke product $\phi$ belongs to $\mathcal{N}$ if and only if the non-tangential limit of $\phi'$ exists almost everywhere on $\partial \D$ and
    $$\int_0^{2\pi} \log^+ \left(\sum_n \frac{1-|z_n|^2}{|e^{i\t}-z_n|^2}\right)d\t<\infty,$$
where $\{z_n\}$ is the zero-sequence of $\phi$.
 \hfill$\Box$

\medskip

For $0<p<\infty$ and $-1<\a<\infty$, the Bergman space $A_\a^p$ consists of those analytic functions in $\D$ such that
    $$
    \|f\|_{A^p_\a}^p=\int_\D|f(z)|^p(1-|z|)^\a\,dm(z)<\infty,
    $$
where $dm(z)$ is the Lebesgue area measure on $\D$. Note that  \cite[Theorem~10]{R2019} implies the inclusion
$$\left\{\Th\in \I:\Th'\in \bigcup_{-1<\a<\infty}\bigcup_{\a+1<p<\infty}A_\a^p\right\} \subset \left\{\Th\in \I: \Th'\in \mathcal{N}\right\}.$$
Hence using Corollary~\ref{Coro2} one can construct one-component inner functions whose derivative does not belong to certain Bergman spaces.

\section{One-component singular inner functions}\label{Sec5}

We begin with another  consequence of Theorem~\ref{Thm1}, which is a slight more general version of Theorem~\ref{singular} stated in the Introduction. 

\begin{corollary}\label{coro-sing}
Let $\Th=BS$, where $B$ is  a Blaschke product with zeros $\{z_n\}$ and $S$ is a singular inner function associated with a non-trivial singular measure $\s$.
Assume that $\{z_n\}\subset \Omega := \{z\in \D: 1-|z|\ge 2\dist (z/|z|, \supp \s)\}$.
Then $\Th\in \I$ if and only if $\limsup_{z\in \Omega, |z|\rightarrow 1^-} |\Th(z)|<1$.
\end{corollary}

\begin{proof}
The necessity  
follows from Theorem~\ref{Thm1} since $Q(z) \cap \supp \s$ is non-empty when $z \in \Omega$.
Hence we only need to prove the sufficiency. By the assumption, there exists a constant $C<1$ such that $|\Th(z)|\le C$ for $z\in \Omega$. Set $K=\{z\in \D: \dist(z,\Omega)\le (1-|z|)/2\}$. Applying the Schwarz-Pick lemma, we find a constant $D < 1$ such that $|\Th(z)|\le D$ for $z\in K$. Next we show that $\Th$ is one-component applying Theorem~\ref{Thm1}. Assume that $|\Th(z)|\ge D$.
Then $z\in \D \setminus K$, and it follows that $Q(z) \cap \Omega =\emptyset$, which implies $\mu(\Th)  (Q(z))=0$. Consequently, $\Th$ is a one-component inner function. 
\end{proof}

Note that the necessity in Corollary~\ref{Coro1b}~or~\ref{coro-sing} follows also from \cite[Lemma~6.1]{Aleksandrov1989} by A.~Aleksandrov,
and this Aleksandrov's result originates from the proof of \cite[Theorem~3]{VT1988} by S.~Treil and A.~Volberg.

\medskip

\noindent
\emph{Proof of Corollary~\ref{countable}.} Let $E=\{\xi_n : n=1,2,\ldots\}$. Then 
$\s=\sum_{n=1}^\infty \a_n \d_{\xi_n}$, where $a_n >0$.  Corollary~\ref{coro-sing} and \cite[Chap.~II, Theorem~6.2]{Garnett1981} show that the singular inner function associated to $\s$ is one-component.  \hfill$\Box$

\medskip

The following example shows that even quite basic singular inner functions may lie out of $\I$.

\begin{example}\label{Ex1}
Let $\{\t_n\}_{n=1}^\infty$ and $\{\a_n\}_{n=1}^\infty$ be sequences of distinct points on $(0,1)$ satisfying
$\lim_{n \rightarrow \infty} \t_n = 0$ and $\sum_{n=1}^\infty \a_n \t_n^{-2}<\infty$.
Set $S$ be the singular inner functions associated with the measure $\s=\sum_{n=1}^\infty \a_n \d_{e^{i\t_n}}$.
Since
    \begin{equation*}
    \begin{split}
    |S(r)|=\exp\left(-\sum_{n=1}^\infty \a_n \frac{1-r^2}{|e^{i\t_n}-r|^2}\right) \ge \exp\left(-3(1-r^2)\sum_{n=1}^\infty \a_n \t_n^{-2}\right) \longrightarrow 1^-, \quad r \rightarrow 1^-,
    \end{split}
    \end{equation*}
the function $S$ does not belong to $\I$ by Corollary~\ref{coro-sing}.

\end{example}

Next we will show that the singular inner function associated to the Cantor measure on a symmetric Cantor set, is one-component.  Let us recall the construction of symmetric Cantor measures $\s$ associated with a sequence $\{\d_n\}$, which were studied for instance in \cite{Ahern1979} by P.~Ahern.
\begin{itemize}
\item Let $\{\d_n\}_{n=0}^\infty$ be a strictly decreasing sequence such that $\d_0=2\pi$ and $\lim_{n\rightarrow \infty} \d_n=0$.
\item Set $E_0=[0,2\pi]$ and $n\in \N$. Define $E_n$ inductively as follows:
$E_n$ consists of $2^{n}$ pairwise disjoint intervals each of length $2^{-n}\d_n$ and $E_{n+1}$ is obtained (from $E_n$) by removing a segment
from each interval of $E_n$. Write $E=\bigcap_{n=0}^\infty E_n$.
\item Define the non-decreasing function $\vp:[0,2\pi] \rightarrow [0,1]$ as follows:
$\vp(0)=0$, $\vp(2\pi)=1$, $\vp$ is a constant on each interval of $[0,2\pi]\setminus E$ and $\vp$ increases by an amount of $2^{-n}$
on each intervals of $E_n$.
\item For $0\le a\le b\le 2\pi$, define the measure $\s$ by $\s((a,b)) = \vp(b)-\vp(a)$.
\end{itemize}
For instance $E$ is the Cantor middle third set if $\d_n=2\pi(2/3)^n$ for all $n\in \N$. 

\begin{corollary}\label{cor3}
If $S$ is a singular inner function associated with a symmetric Cantor measure, then $S \in \I$.
\end{corollary}

\medskip

\noindent
\emph{Proof of Corollary~\ref{cor3}.} Let $\{\d_n\}$ be the sequence associated with the symmetric Cantor measure $\s$ that induces $S$, and let $\Omega$ be as in Corollary~\ref{coro-sing}.
Assume that $z\in \Omega$ and $2^{-n}\d_n\le 1-|z|\le 2^{-(n-1)}\d_{n-1}$ for some $n\in \N \setminus \{1,2\}$.
Then $4Q(z)$ contains an interval of $E_n$ from the construction of $\s$.
Hence
    \begin{equation*}
    \begin{split}
    P[\s](z)&=\int_0^{2\pi} \frac{1-|z|^2}{|z-e^{i\t}|^2}\,d\s(\t)
    \ge \int_{e^{i\t}\in 4Q(z)} \frac{1-|z|^2}{|z-e^{i\t}|^2}\,d\s(\t) \\
    &\gtrsim \frac{\s(4Q(z)\cap \partial \D)}{1-|z|}\ge (2 \d_{n-1})^{-1} \longrightarrow \infty, \quad n \rightarrow \infty.
    \end{split}
    \end{equation*}
Consequently, we obtain
    $$\lim_{z\in \Omega, |z|\rightarrow 1^-} |S(z)|=\lim_{z\in \Omega, |z|\rightarrow 1^-} \exp\left(-P[\s](z)\right)=0,$$
and the assertion follows from Corollary~\ref{coro-sing}. \hfill$\Box$

\medskip

Next we present another consequence of Corollary~\ref{coro-sing}. 

\begin{corollary}\label{density}
Let $\s$ be a positive singular measure in the unit circle and let $S$ be the corresponding singular inner function. Assume that there exists a constant $\delta >0$ such that for any point $\xi$ in the closed support of $\s$ we have
$$
\liminf_{h \to 0^+ } \frac{\s (\{\psi \in \partial \D : |\psi - \xi| < h \})}{h} > \delta. 
$$
Then $S$ is a one-component inner function.
\end{corollary}

\medskip

\noindent
\emph{Proof of Corollary~\ref{density}}.
Given a point $z \in \D \setminus \{0\}$ let $I(z)$ be the arc centered at $z/|z|$ of length $1-|z|$. Note that there exists an absolute constant $C>0$ such that for any $z \in \D \setminus \{0 \}$, we have 
\begin{equation*}
\int_0^{2\pi} \frac{1-|z|^2}{|z-e^{i\t}|^2}\,d\s(\t)
    \ge C \frac{\s (I(z))}{1-|z|}. 
\end{equation*}
Hence the assumption implies that for any point $\xi$ in the support of $\s$ we have 
\begin{equation*}
\liminf_{r \to 1^-} \int_0^{2\pi} \frac{1-|r|^2}{|r \xi-e^{i\t}|^2}\,d\s(\t)
    \ge C \delta. 
\end{equation*}
Since the previous integral is $- \log|S(r \xi)|$,  Corollary~\ref{coro-sing} implies that $S$ is one-component. \hfill$\Box$

\medskip

We finish the paper with an open question. Level sets of bounded analytic functions are related to the original proof of the Corona Theorem. It is well known that there exists a bounded analytic function having all level sets of infinite length. More concretely P. Jones constructed an analytic function $\Th$ from the unit disc into itself such that for any $0<c<1$, the level set $\{z \in \D : |\Th(z)|= c \}$ has infinite length. See \cite{Jones1980}. We say that an inner function $\Th$ has property $(A)$ if there exist an inner function $B$ and a constant $0<c<1$ such that arc length on the level set $\{z \in \D : |\Th(z) B(z)|= c \}$ is a Carleson measure. Roughly speaking, an inner function has property $(A)$ if by adding more zeros, one can produce a nice level set. 

\begin{question}\label{question}
Does property $(A)$ hold for any inner function?
\end{question}

 B. Cohn proved that for any one-component inner function $\Th$, there exists $0<c<1$ such that the arc length of the level set $\{z \in \D : |\Th (z)|= c \}$ is a Carleson measure (\cite{Cohn1982}). Hence our Theorem~\ref{Thm2} gives that any inner function whose singular set has measure zero, has property $(A)$. It is likely that using the techniques in \cite{Garnett-Nicolau1996} one could factor any inner function $\Th$ into a finite number of inner functions $\Th_i$ having property $(A)$.

\end{document}